\documentclass[dvipsnames]{amsart}
\usepackage[utf8]{inputenc}
\usepackage{graphicx}
\usepackage{amssymb}
\usepackage{amsmath,tensor}
\usepackage{float}
\usepackage{xcolor}
\usepackage[
backend=biber,
style=alphabetic,
maxnames = 5
]{biblatex}
\usepackage{mathscinet}

\usepackage{hyperref}
\usepackage{cleveref}
\usepackage{comment}
\usepackage[all]{xy}
\usepackage{tikz-cd}

\usepackage[dvips]{epsfig}
\usepackage{pinlabel}
\newcommand{\ig}[2]{\vcenter{\xy (0,0)*{\includegraphics[scale=#1]{./EPS/#2}} \endxy}}

\newtheorem{thm}{Theorem}[section] 
\newtheorem{prop}[thm]{Proposition}
\newtheorem{lem}[thm]{Lemma}
\newtheorem{cor}[thm]{Corollary}

\theoremstyle{definition}
\newtheorem{defn}[thm]{Definition}

\newtheorem{ex}[thm]{Example}

\theoremstyle{remark}
\newtheorem{rem}[thm]{Remark}

\def\Z{\mathbb{Z}}

\def\hat{\widehat}

\newcommand{\co}{\colon}
\newcommand{\pa}{\partial}

\DeclareMathOperator{\Hom}{Hom}
\DeclareMathOperator{\End}{End}

\DeclareMathOperator{\id}{id}

\newcommand{\inv}{^{-1}}

\newcommand{\at}{{\mathtt{a}}}

\newcommand{\AC}{\mathcal{AC}}
\newcommand{\AD}{\mathcal{AD}}

\newcommand{\mi}{\underline}
\newcommand{\ma}{\overline}
\newcommand{\expr}{\leftrightharpoons}

\newcommand{\Dem}{\mathcal{D}}

\newcommand{\SC}{\mathcal{SC}}
\newcommand{\mt}{\emptyset}

\DeclareMathOperator{\leftdes}{LD}
\DeclareMathOperator{\rightdes}{RD}
\DeclareMathOperator{\leftred}{LR}
\DeclareMathOperator{\rightred}{RR}
\DeclareMathOperator{\core}{core}

\DeclareMathOperator{\nilcox}{LNC}

\newcommand{\squish}{\mathbb{S}q}
\newcommand{\calS}{\mathcal{S}}

\newcommand{\revise}[1]{{ #1}}

\title{On reduced expressions for core double cosets}
\author[Ben Elias]{Ben Elias}
\address{Department of Mathematics, Fenton Hall, University of Oregon,
Eugene, OR, 97403-1222, USA}
\email{belias@uoregon.edu}

\author[Hankyung Ko]{Hankyung Ko}
\address{Department of Mathematics, Uppsala University,
Box. 480,
SE-75106, Uppsala, Sweden}
\email{hankyung.ko@math.uu.se}

\author[Nicolas Libedinsky]{Nicolas Libedinsky}
\address{Departamento de Matemáticas, Facultad de Ciencias, Universidad de Chile}
\email{nlibedinsky@gmail.com}

\author[Leonardo Patimo]{Leonardo Patimo}
\address{Dipartimento di Matematica, Universit\`a di Pisa, Italy}
\email{leonardo.patimo@unipi.it}

\begin{document}

\begin{abstract}
The notion of a reduced expression for a double coset in a Coxeter group was introduced by Williamson, and recent work of Elias and Ko has made this theory more accessible and combinatorial. One result of Elias--Ko is that any coset admits a reduced expression which factors through a reduced expression for a related coset called its core. In this paper we define a class of cosets called atomic cosets, and prove that every core coset admits a reduced expression as a composition of atomic cosets. This leads to an algorithmic construction of a reduced expression for any coset. In types $A$ and $B$ we prove that the combinatorics of compositions of atomic cosets matches the combinatorics of ordinary expressions in a smaller group. In other types the combinatorics is new, as explored in a sequel by Ko.

\end{abstract}

\maketitle

\section{Introduction}

\subsection{Double cosets and reduced expressions}

Let $(W,S)$ be a Coxeter system, with simple reflections $\{s_i\}_{i \in S}$. For any subset $I \subset S$, one can consider the parabolic subgroup $W_I$ generated by $\{s_i\}_{i \in I}$. For $I, J \subset S$, double cosets in $W_I \backslash W / W_J$ are a subject of common study in algebra and combinatorics. We call them \emph{$(I,J)$-cosets}. In this paper we only treat the case when $W_I$ and $W_J$ are finite groups, in which case $I$ and $J$ are called \emph{finitary}. In this case, any double coset $p$ has a unique maximal element $\ma{p}$ and minimal element $\mi{p}$ with respect to the Bruhat order \revise{ (see \cite[Prop. 31]{Rea06}).} Let $w_I$ denote the longest element of $W_I$, and let $\ell(I)$ denote its length.

The concept of a reduced expression for a double coset was introduced in \cite{WillThesis}. Many equivalent and more user-friendly definitions were given in \cite{EKo}. A \emph{(multistep) $(I,J)$-expression} is a sequence of finitary subsets of the form 
\begin{equation} M_{\bullet} = [[I = I_0 \subset K_1 \supset I_1 \subset K_2 \supset \ldots \subset K_m \supset I_m = J]]. \end{equation}
(Note that $I_0 = K_1$ and other equalities are permitted.) This expression $M_{\bullet}$ is a reduced expression for $p$ if and only if
\begin{equation} \ma{p} = w_{K_1} w_{I_1}^{-1} w_{K_2} \cdots w_{K_m} \end{equation} 
and
\begin{equation} \ell(\ma{p}) = \ell(K_1) - \ell(I_1) + \ell(K_2) - \ldots + \ell(K_m). \end{equation}
If $M_{\bullet}$ is an $(I,J)$-expression and $N_{\bullet}$ is a $(J,K)$-expression, then the concatenation $M_{\bullet} \circ N_{\bullet}$ (defined in the intuitive way) is an $(I,K)$-expression.

Many properties of reduced expressions and tools for studying them were developed in \cite{EKo}. Of particular note are the \emph{(double coset) braid relations}, which are certain transformations between expressions, and the double coset Matsumoto theorem, which states that any two reduced expressions are related by braid relations.

What is absent from \cite{EKo} is a practical matter: guidance on how to construct reduced expressions for double cosets. This is what we aim to rectify in this paper, with the introduction of atomic expressions. We feel that atomic expressions will be an interesting new gadget for combinatorial study, a study which is begun here and continued in \cite{paper6}. 

\subsection{Core cosets and redundancy}

Let $p$ be an $(I,J)$-coset.

\begin{ex} In our running example, $W = \calS_{11}$ will be the symmetric group on $11$ letters, with parabolic subgroups $W_I = S_4 \times S_4 \times S_3$ and $W_J = S_2 \times S_5 \times S_1 \times S_2 \times S_1$. Then $p$ will be the following double coset.
\begin{equation} \ig{1}{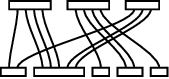} \end{equation}
The string diagram represents a reduced expression for the minimal element $\mi{p}$, while the boxes represent $W_J$ (on the bottom) and $W_I$ (on the top).\end{ex}

The left and right redundancy sets of $p$ are defined and denoted as
\begin{equation} \leftred(p) = I \cap \mi{p} J \mi{p}^{-1}, \qquad \rightred(p) = \mi{p}^{-1} I \mi{p} \cap J. \end{equation}

\begin{rem}\label{LR=RR}
 Conjugation by $\mi{p}$ induces an isomorphism  \[(W_{\leftred(p)},\leftred(p))\cong (W_{\rightred(p)},\rightred(p)) \]
of Coxeter systems.
\end{rem}

In type $A$, the redundancies represent the crossings of two consecutive strands which one could add either on top or on bottom of the diagram while staying within the same double coset. Let us clarify with an example.
 
\begin{ex} \label{ex:runningredunpic} In our running example, 
\[ \leftred(p) = \{s_2,s_3, s_6, s_{10}\}, \qquad \qquad  \rightred(p) = \{s_3, s_4, s_6, s_9\}.\] The following equality indicates  simultaneously why  $s_2 \in \leftred(p)$ and $s_3 \in \rightred(p)$.  
\begin{equation} s_2 \mi{p} = \ig{1}{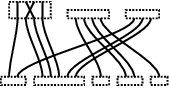} \; = \; \ig{1}{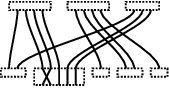} \; = \mi{p} s_3 . \end{equation}
Note that $s_8 \notin J$, so even though $s_5 \mi{p} = \mi{p} s_8$, we have $s_5 \notin \leftred(p)$ and $s_8 \notin \rightred(p)$. \end{ex}

Let $K = \leftred(p)$ and $L = \rightred(p)$. There exists a unique $(K,L)$-coset $p^{\core}$ with $\mi{p^{\core}} = \mi{p}$. We call this the \emph{core} of $p$. Note that $K = \leftred(p^{\core})$ and $L = \rightred(p^{\core})$. More generally, an $(I,J)$-coset $q$ is a \emph{core coset} if $I = \leftred(q)$ and $J = \rightred(q)$.

\begin{ex} Here is a picture of $p^{\core}$ in our running example. The red lines indicate simple reflections which were removed from $I$ or $J$ to reach the redundancy.
\begin{equation} \ig{1}{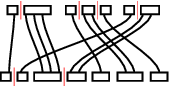} \end{equation} \end{ex}

A useful result from \cite{EKo} is the statement that $p$ has a reduced expression which factors through a reduced expression for $p^{\core}$. Namely, if $M_{\bullet}$ is a reduced expression for $p^{\core}$, then
\begin{equation} [[I \supset K]] \circ M_{\bullet} \circ [[L \subset J]] \end{equation}
is a reduced expression for $p$. This reduces the problem of constructing reduced expressions to the case of core cosets.

Note that $I$ and $J$ need not have the same size for general double cosets, but they do for core cosets, \revise{as they are conjugate}.

\begin{rem} There are many other reduced expressions of $p$ which do not factor through $p^{\core}$. They are of practical and combinatorial interest, but we do not discuss them in this paper. \end{rem}

\begin{rem} Reduced expressions for $p$ which factor through the core play a particularly important role in the theory of singular Soergel bimodules and their relationship to standard bimodules, and consequently they play a crucial role in the algorithm which constructs a basis for morphism spaces between singular Soergel bimodules. See \cite{KELP4} for more details. \end{rem}

\subsection{Atomic cosets and atomic expressions}

We now introduce the main new concepts in this paper.

Let $M \subset S$ be finitary and $s \in M$, and let $t = w_M s w_M$. Let $I = \hat{s} := M \setminus s$ and $J = \hat{t} := M \setminus t$. Let $p$ be the $(I,J)$-coset containing $w_M$. Then $p$ is a core coset. Any coset of this form will be called an \emph{atomic coset}.

\begin{ex} \label{ex:typeAatomic} Here are two prototypical examples of atomic cosets in type $A$. \begin{equation} \ig{1}{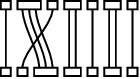} \qquad \qquad \qquad \ig{1}{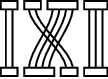} \end{equation}
\revise{In each example there are two blocks which cross each other. If $I = M \setminus s$ and $J = M \setminus t$ as above, then $s$ is the crossing between these two blocks on top, and $t$ the crossing between them on bottom, so that in $M$ the two blocks are merged. In the first example, $s = s_4$ and $t = s_2$.}
\end{ex}

It is proven in \cite[Proposition 5.14]{EKo} that an atomic coset has a unique reduced expression, namely $[[I \subset M \supset J]]$, which we shorten to $[I, M, J]$ or $[I + s - t]$, using the different notations introduced in \cite{EKo}.  A \emph{reduced atomic expression} is a reduced expression which is a composition of reduced expressions for atomic cosets.

\begin{thm} \label{corehasatomic} Any core coset has a reduced atomic expression. \end{thm}

\begin{ex}\label{ex:squashing} A reduced atomic expression for $p^{\core}$ in our running example is 
\begin{equation} [I + s_8 - s_8 + s_9 - s_{10} + s_7 - s_6 + s_8 - s_8 + s_5 - s_5 + s_6 - s_7 + s_4 - s_2]. \end{equation} Here is the picture, with colors for visual emphasis. \begin{equation} \ig{1}{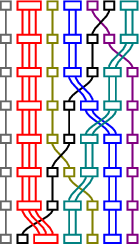} \end{equation} In this picture, we drew the minimal elements in each atomic coset, and composed them to obtain the minimal element of $p^{\core}$. This is not how typical cosets compose, but it is how core cosets compose, see Lemma~\ref{lem:mimimi}. \end{ex}

In this paper we study atomic expressions, we develop some of the basic properties thereof, and we develop an algorithm to find them.


\begin{ex} Let $I = J = \mt$. For any $w \in W$, the $(\mt,\mt)$-coset $\{w\}$ is a core coset. The atomic $(\mt,\mt)$-cosets are $\{s\}$, where $s \in S$ is a simple reflection. The unique reduced expression for $\{s\}$ is $[[\mt \subset \{s\} \supset \mt]]$, shortened to $[\mt,s,\mt]$. There is a bijection between ordinary reduced expressions $s_1 s_2 \ldots s_d$ for $w$, and atomic reduced expressions $[\mt,s_1,\mt,s_2,\mt, \cdots, s_d, \mt]$ for $\{w\}$. \end{ex}

\begin{ex} An important class of atomic expressions has already appeared in \cite{EKo}, though we did not think of them as atomic expressions at the time. They appear within the \emph{switchback relation}, which is one of the aforementioned braid relations. Suppose $I$ is finitary and $s, t \in I$ (possibly equal), but $t \ne w_I s w_I$. Then $[[\hat{s} \subset I \supset \hat{t}]]$ is reduced, but not atomic or core. The switchback relation gives another reduced expression for the same double coset, having the form
\begin{equation} [[\hat{s} \supset K]] \circ A_{\bullet} \circ [[L \subset \hat{t}]]. \end{equation}
Here $A_{\bullet}$ is an atomic expression for the core coset. For an explicit description of $A_{\bullet}$ in each case, see \cite[Section 6]{EKo}.
\end{ex}

\begin{rem} A core coset can have many reduced expressions which are not atomic expressions. Again, these are of interest, but we do not focus on them in this paper. \end{rem}

\begin{rem} Theorem \ref{corehasatomic} will be used to enable a polynomial forcing result for singular Soergel bimodules in future work. \end{rem}

\subsection{Atomic reduced expressions in type \texorpdfstring{$A$}{A}}

From the example above, the reader may already see how to construct an atomic expression for a core coset in type $A$. For a core coset $p$, the strands within each irreducible component of the redundancy are bundled together and never leave each others' side (they have the same color in the picture above). Consequently, one can squash them into a single strand, to obtain a permutation $w$ in a smaller symmetric group.

\begin{thm} Let $(W,S)$ have type $A_{n-1}$ so that $W \cong \calS_n$. Let $J \subset S$ have size $m$. Then there is a bijection $\squish^J$ between core $(I,J)$-cosets (as $I$ varies) and permutations in $\calS_{n-m}$. For a core $(I,J)$-coset $p$, $\squish^J$ induces a bijection between atomic reduced expressions for $p$ and ordinary reduced expressions for $\squish^J(p)$. \end{thm}


\begin{ex}\label{ex:squash} Here is the result of squashing in our running example.

\begin{equation} \squish \co \quad \ig{1}{examplecoreatomicrex} \quad \mapsto \quad \ig{1}{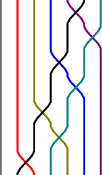} \end{equation} \end{ex}

We hope the construction of $\squish$ is clear from the example above. For the proof and the details see Theorem~\ref{thm.psi}.  A similar squashing result holds in type $B$, see Section~\ref{s.typeB}.  We explicitly describe how every ordinary braid relation corresponds under $\squish^{-1}$ to a composition of several double coset braid relations between atomic expressions (which pass through non-atomic expressions en route). See Theorem~\ref{thm.atomicMatsumoto} and the preceding discussion for details.


This result is positive in the sense that it gives a complete understanding of atomic reduced expressions in types $A$ and $B$, in terms of existing combinatorics. 
This is extremely useful for practical purposes, such as finding reduced expressions of double cosets. 
Once the right perspective is in place, the result is relatively straightforward, but before this paper it was not known how to efficiently find reduced expressions for double cosets, nor was there any expectation that it would be easy.

The result is negative for combinatorial thrill-seekers, in that no new combinatorics arises in types $A$ and $B$. However, the case of a general Coxeter group is quite different. That is, the set of atomic expressions for a core coset is not related to reduced expressions for some other Coxeter group, but gives rise to a new combinatorial structure. This is studied in the follow-up paper \cite{paper6}.

\subsection{Non-reduced atomic expressions}

In ordinary Coxeter theory there are many variants of the group algebra of $W$, which go under the name of \emph{generalized Hecke algebras}. All of them have a basis parametrized by $W$. These include among others the Iwahori-Hecke algebra, the nilHecke algebra, and the $0$-Hecke algebra which is the linearization of the $*$-monoid. These agree in how they treat reduced expressions, but disagree on how to interpret non-reduced expressions. Similarly, there are multiple categories one can build whose objects are $I \subset S$ finitary, and where $\Hom(J,I)$ has a basis parametrized by $(I,J)$-cosets. For the discussion below we focus on $\SC$ (Section \ref{s:singmon}),
 which corresponds to the $*$-monoid, and on $\Dem$, which corresponds to the nilHecke algebra, see \cite{paper1} for more details.

It is interesting to consider the non-full subcategory $\AC \subset \SC$ or $\AD \subset \Dem$ generated by atomic cosets, and ask for a presentation thereof. One way of phrasing our results from the previous section is that, in types $A$ and $B$, there are braid relations in $\AC$ and $\AD$, which become the ordinary braid relations after squashing.

A nice feature of the nilHecke algebra, or its double coset analogue $\Dem$, is that any non-reduced expression is equal to zero. This property is inherited by $\AD$. Therefore, in type $A$, squashing induces an equivalence of categories between $\AD$ and a nilHecke algebroid. See \Cref{cor.AC0} for a precise statement. A similar theorem holds in type $B$.

In the follow-up paper \cite{paper6}, the second author gives a presentation for $\AD$ for any Coxeter type. The same problem for $\AC$ seems more difficult and is open. See \cite[Section 6]{paper6} for some atomic non-braid relations in $\AC$ and more discussion.

\subsection{Background reference}

The background section found in \cite[Chapter 2]{paper1} suffices for all the needs of this paper. We refer the reader to that paper for basic material on double cosets, expressions for double cosets, and basic notation.

\revise{\subsection{Organization}
In \Cref{sec:atomic} we introduce and study core cosets and atomic cosets in general for any Coxeter group. The main results of this section are \Cref{thm:corestartswithatomic} and \Cref{cor:coresplitsintoatomic}, which state that any core coset admits an atomic reduced expression. In \Cref{subsec:atomicsubcats} we introduce the atomic subcategories $\AC$ and $\AD$. In \Cref{sec:typeA} we focus on type $A$, where the combinatorics of core cosets can be related via squashing, defined in \Cref{s:w(p)}, to the combinatorics of a smaller symmetric group (\Cref{thm.atomicrex}). Applying this, we prove an atomic version of the Matsumoto theorem (\Cref{thm.atomicMatsumoto}) and give a presentation of the atomic category $\AD$. In \Cref{s.typeB} we prove the analogous results in type $B$ (\Cref{thm:atomicrexB} and \Cref{prop:ADpresentationB}).}

 \section*{Acknowledgements}  BE was supported by NSF grants DMS-2201387 and DMS-2038316. HK was partially supported by the Swedish Research Council. NL was partially supported by FONDECYT-ANID grant 1230247. \revise{We thank the referee for helpful comments.}

\section{Atomic cosets and expressions}\label{sec:atomic}

\subsection{Coxeter monoids}\label{s:singmon}

We recall the regular and singular Coxeter monoids. See \cite[Section 1.3, Section 2]{EKo} for details.

Let $(W,S)$ be a Coxeter system. The \emph{Coxeter monoid} $(W,S,*)$ (also known as the \emph{0-Hecke monoid} or the \emph{star monoid}) is the monoid generated by $S$ with relations 
\begin{subequations}
\begin{equation} \label{starquad} s*s = s \text{ for each } s \in S, \end{equation}
\begin{equation} \label{starbraid} \underbrace{s * t * \cdots}_{m} = \underbrace{t * s * \cdots}_{m}, \text{ for each } s, t \in S \text{ with } m = m_{st} < \infty. \end{equation} 
\end{subequations}
The relation \eqref{starquad} is called the \emph{$*$-quadratic relation}, and \eqref{starbraid} is the \emph{braid relation}.

If $s*t*\cdots *u$ is a reduced expression in $(W,*)$, then $st\cdots u$ is a reduced expression in the Coxeter group $W$, and we identify the element $w=s*t*\cdots *u\in (W,*)$ with the element $w=st\cdots u\in W$. More generally, if $\ell(xy)=\ell(x)+\ell(y)$, where $\ell$ denotes the Coxeter length, i.e., the length of a reduced expression, then we both write $x*y=x.y$ in $(W,*)$ and $xy=x.y$ in $W$, and we call this a \emph{reduced composition}.

The \emph{singular Coxeter monoid} is the category $\SC=\SC(W,S)$ defined as follows. Its objects are the finitary subsets $I \subset S$. Morphisms from $J$ to $I$ are identified with the set of $(I,J)$-cosets, i.e.
\[ \Hom_\SC(J,I) = W_I \backslash W/W_J. \]
Composition $*$ of cosets is defined so that $\ma{p*q} = \ma{p}*\ma{q}$. That is, for $p\in W_I \backslash W/W_J$ and $q\in W_J \backslash W/W_K$ we have
\[p*q = W_I (\ma{p}*\ma{q}) W_K.\]
 We say that $p*q$ is reduced if \[ (\ma{p}w_J^{-1}).\ma{q}=\ma{p}*\ma{q} = \ma{p}.(w_J^{-1} \ma{q}).\]
In this case we write $p.q=p*q$. Even though $w_J$ is an involution, we write $w_J^{-1}$ to emphasize that this product lowers the length accordingly.

\revise{Throughout this paper we write $Is$ to denote $I \cup \{s\}$, when $s \in S \setminus I$ and $Is$ is finitary.}



\subsection{Reduced expressions, descent sets, and redundancy}

 For a double coset $p$, let $\leftdes(p)$ and $\rightdes(p)$, the left and right descent sets of $p$,
denote the (ordinary) left and right descent sets of $\overline{p}$. Because we use it often, let us recall and rephrase a result from \cite{EKo}. 
\begin{thm} \label{thm:addremove} Let $p$ be an $(I,J)$ coset, and $s \in S$. There is a reduced expression for $p$ \dots \begin{itemize}
\item \dots of the form $[I,Is,\ldots,J]\phantom{J \setminus s}$ if and only if $s \in \leftdes(p)\setminus I$.
\item \dots of the form $[I,I \setminus s, \ldots, J]\phantom{Js}$ if and only if $s\in I\setminus\leftred(p)$.
\item \dots of the form $[I, \ldots, Js,J]\phantom{I \setminus s}$ if and only if $s \in \rightdes(p)\setminus J$.
\item \dots of the form $[I, \ldots, J\setminus s, J]\phantom{Is}$ if and only if $s \in  J\setminus  \rightred(p)$.
\end{itemize}
\end{thm}

\begin{proof} This is a conglomeration of the results from \cite[Sections 4.8 and 4.9]{EKo}, see also \cite[Theorems 2.20 and 2.22]{paper1}. \end{proof}

\subsection{Core cosets}

\begin{defn}\label{core} Let $p$ be an $(I,J)$-coset and set $K=\leftred(p)$ and $L = \rightred(p)$. The \emph{core} of $p$ is the $(K,L)$-coset $p^{\core}$ with minimal element $\mi{p^{\core}}=\mi{p}$. \end{defn}

\begin{defn} Let $p$ be an $(I,J)$-coset. We call $p$ a \emph{core coset} if $\leftred(p) = I$ and $\rightred(p) = J$, or equivalently if $I = \mi{p} J \mi{p}^{-1}$. \end{defn}

\begin{lem}\label{coreiscore} If $p$ is any double coset, then $p^{\core}$ is a core coset. If $p$ is a core coset, then $p = p^{\core}$. \end{lem}

\begin{proof} The first statement is a consequence of \cite[Lemma 4.27]{EKo}. The second statement follows from the definitions. \end{proof}

\begin{ex} Any $(\mt, \mt)$-coset is a core coset. \end{ex}

\begin{ex} The $(I,I)$-coset containing the identity is a core coset, for any $I$.  \end{ex}

\begin{ex} Let $W = S_3$ and $S = \{s,t\}$. Aside from the core cosets guaranteed by the previous examples, the only core cosets are the $(s,t)$-coset and the $(t,s)$-coset containing the longest element. \end{ex}

 As in earlier works, we write $p \expr I_{\bullet}$ whenever $I_{\bullet}$ is an expression for $p$, \revise{or $I_{\bullet} \expr I'_{\bullet}$ if they both express the same coset}.

\begin{ex}\label{ExS4} Let $W = S_4$ and $S = \{s,t,u\}$ with $m_{su} = 2$. Consider the $(s,u)$-coset $p$ with 
\begin{equation} \ma{p} = tstut, \quad \mi{p} = tsut, \qquad p \expr [s,st,t,tu,u]. \end{equation}
Consider the $(t,s)$-coset $q$ with
\begin{equation} \ma{q} = stsu, \quad \mi{q} = stu,  \qquad q \expr [t,st,s,su,s]. \end{equation}
These are both core cosets. \end{ex}

In the next few lemmas we give several basic properties of core cosets.

\begin{lem} \label{lem:coreiffminp} Let $p$ be an $(I,J)$-coset. Then $p$ is a core coset if and only if \begin{equation} \label{mapforcore} \ma{p} = w_I . \mi{p} = \mi{p} . w_J. \end{equation}
\end{lem}

\begin{proof} That equation \eqref{mapforcore} holds if and only if $\leftred(p) = I$ and $\rightred(p) = J$ is immediate from 
Howlett's lemma \cite[Lemma 2.12 (3)]{EKo}. 
\end{proof}

The following lemma is a crucial technical tool, and fails when $p$ and $q$ are not core cosets.

\begin{lem}\label{lem:mimimi}
    Let $p$ be a core $(I,J)$-coset and $q$ a core $(J,K)$-coset. Then $p*q$ is reduced if and only if $\mi{p} * \mi{q}$ is reduced. If so, then 
\begin{equation}\label{eq:mimimi}
    \mi{p*q}=\mi{p}.\mi{q}
\end{equation}
holds, and $p*q$ is core.
\end{lem}
\begin{proof}
 Recall that by definition $p * q$ is reduced if $(\ma{p} w_J^{-1})*\ma{q}$ is a reduced composition. By Lemma~\ref{lem:coreiffminp}, $\ma{p} w_J^{-1} = \mi{p}$ and $\ma{q} = \mi{q} . w_K$. Thus $p * q$ is reduced if and only if $\mi{p} * \mi{q} * w_K$ is a reduced composition. Certainly this requires the subword $\mi{p} * \mi{q}$ to be a reduced composition.


Now suppose $\mi{p} * \mi{q}$ is reduced, but $\mi{p}*\mi{q} * w_K$ is not. Then $\rightdes(\mi{p} * \mi{q})\cap K$ is nonempty.
Let $s\in \rightdes(\mi{p} * \mi{q})\cap K$. By the exchange property (see e.g., \cite[Theorem~1.4.3]{BjornerBrenti}), either the product $\mi{q}*s$ or  $\mi{p}*t$ is not reduced, where $t= \mi{q}s\mi{q}\inv\in \mi{q}K\mi{q}\inv = J$. This contradicts the minimality of $\mi{p}$ in $p$ or $\mi{q}$ in $q$. Thus $\mi{p} * \mi{q}$ being reduced implies that $\mi{p}*\mi{q} * w_K$ and  hence $p * q$ is reduced.


Suppose $p*q$ is reduced. Since $\ma{p*q} = \ma{p}*\ma{q} = \mi{p}.\mi{q}.w_K$ as above, we deduce that $\mi{p}.\mi{q}$ and $\ma{p*q}$ are in the same $(I,K)$-coset $p*q$. We've already shown $\rightdes(\mi{p}.\mi{q}) \cap K$ is empty. An analogous argument proves $\leftdes(\mi{p}.\mi{q}) \cap I$ is empty. Thus $\mi{p}.\mi{q}$ is a minimal coset representative, and $\mi{p}.\mi{q} = \mi{p*q}$.

Thus $\ma{p*q} = \mi{p*q} . w_K$, implying by Lemma~\ref{lem:coreiffminp} that $p*q$ is core. \end{proof}

\begin{lem} \label{lem:samesizeeasyshowcore} Let $p$ be an $(I,J)$-coset, where $I$ and $J$ have the same size. If $\leftred(p) = I$ or $\rightred(p) = J$, then $p$ is a core coset. \end{lem}

\begin{proof} Suppose that $\leftred(p) = I$. The right redundancy of $p$ is a subset of $J$, but has the same size as $I$, hence it must be $J$.
\end{proof}

For the next lemma, recall from \cite[Section 4.4]{EKo} that if $I_{\bullet} = [I_0, \ldots, I_d]$ then $I_{\bullet}^{-1} = [I_d, \ldots, I_0]$ denotes the same expression in reversed order. By \cite[Proposition 4.7]{EKo}, if $I_{\bullet} \expr p$, then $I_{\bullet}^{-1} \expr p^{-1}$ where $p^{-1} = \{w^{-1} \mid w \in p\}$. Note that $p * p^{-1}$ is not usually the identity $(I,I)$-coset; only the identity coset is invertible under $*$-multiplication.

\begin{lem} \label{lem:reverseword} If $p$ is a core $(I,J)$-coset, then $p^{-1}$ is a core $(J,I)$-coset. \end{lem}

\begin{proof} Note that $\mi{(p^{-1})} = \mi{p}^{-1}$. Clearly $I = \mi{p} J \mi{p}^{-1}$ is equivalent to $J = \mi{p}^{-1} I \mi{p}$. \end{proof}

Finally, in preparation for the next section, we recall another result from \cite{EKo}. Recall that for any finitary subset $J \subset S$, conjugation by the longest element $w_J \in W_J$ induces an automorphism of $W_J$ which permutes its simple reflections $J$. For simplicity we do not write inverses on longest elements below.

\begin{lem}\label{lem:atomic} Let $I, J \subset S$ and $s, t \in S$ be such that $Is = Jt$ is finitary. We allow $s = t$ and $I = J$. Then $[I+s-t] = [I, Is, J]$ is always a reduced expression for the $(I,J)$-coset containing $w_{Is}$. It is the reduced expression of a core coset if and only if $w_{Is} t w_{Is} = s$, in which case it is the unique reduced expression for that coset. \end{lem}

\begin{proof} This is effectively a restatement of \cite[Lemma 5.22]{EKo}. For completeness, we spell out some parts of the proof again.

Let $[I,Is,J] \expr p$. Then $\ma{p} = w_{Is}$ and it is straightforward from the definition that $[I,Is,J]$ is reduced. We wish to show that $p$ is core if and only if $w_{Is} t w_{Is} = s$, which holds if and only if $w_{Is} J w_{Is} = I$.

If $p$ is a core coset, then $\mi{p} = w_{Is} w_J$, by \eqref{mapforcore}, and
\begin{equation} I = \mi{p} J \mi{p}^{-1} = w_{Is} w_J J w_J w_{Is} = w_{Is} J w_{Is}. \end{equation}

Meanwhile, suppose that $w_{Is} J w_{Is} = I$.  Then $w_{Is} W_J = W_I w_{Is}$. Since the underlying set of the double coset $p$ is $W_I w_{Is} W_J = w_{Is} W_J$, it has minimal element $\mi{p} = w_{Is} w_J$. By Lemma \ref{lem:coreiffminp}, $p$ is a core coset.

The statement that $[I+s-t]$ is the unique reduced expression in this case is \cite[Proposition 5.14(1)]{EKo}. \end{proof}

\subsection{Atomic cosets and reduced expressions for core cosets}\label{s.atom}

Now we address the decomposition of reduced expressions for core cosets.

\begin{prop} \label{prop:splitcore} Let $r$ be a core $(I,K)$-coset, and suppose that $r = p.q$ where $p$ is an $(I,J)$-coset and $q$ is a $(J,K)$-coset. If $J$ has the same size as $I$ and $K$, then $p$ and $q$ are also core cosets. \end{prop}

\begin{proof} By \cite[Lemma 4.22]{EKo}, the left redundancy sequence of a reduced expression is weakly decreasing, see \cite[Definition 2.17]{EKo}. The immediate consequence is that $\leftred(r) \subset \leftred(p)$. But $\leftred(r) = I$ and $\leftred(p) \subset I$ so $\leftred(p) = I$. By Lemma \ref{lem:samesizeeasyshowcore}, $p$ is core.

Reversing the order of the words we have $r^{-1} = q^{-1} . p^{-1}$. Then $r^{-1}$ is a core coset by Lemma \ref{lem:reverseword}, and $q^{-1}$ is a core coset by the previous paragraph. Reversing again, $q$ is a core coset. \end{proof}

Suppose that $[I = I_0, \ldots, I_d = K]$ is a reduced expression for a core coset $r$. None of the subsets $I_k$ has size less than $I$, by \cite[Lemma 4.22]{EKo}. If $I_k$ has the same size as $I$ and $K$ for some $1 \le k \le d-1$, then Proposition \ref{prop:splitcore} enables us to factor $r$ (through $I_k$) as a reduced composition $p.q$ of smaller core cosets. Not every reduced expression for $r$ will admit such a factorization, as the next example shows. But the following theorem and corollary state that some reduced expression for $r$ will factor into small pieces.

\begin{ex} We continue with the notation given in Example \ref{ExS4}. Consider the reduced expression $[\mt, s, st, t, tu, u, \mt]$ for the $(\mt, \mt)$-coset $\{tstut\}$. This expression does not factor as a composition of reduced expressions for core cosets. \end{ex}

\begin{thm} \label{thm:corestartswithatomic} Let $r$ be a core $(I,J)$-coset. There are two cases. \begin{itemize}
\item $r$ is an \emph{identity coset}, meaning $I = J$ and $\mi{r} = \id$.  Then $[I]$ is the unique reduced expression for $r$.
\item $r = p.q$ for core cosets $p$ and $q$, where $p = [I + s - t]$ for some $s \in S \setminus I$, where $t = w_{Is} s w_{Is}$. Moreover, $[I + s - t]$ is the unique reduced expression for $p$. We allow the possibility that $r = p$ and $q$ is an identity coset.
\end{itemize}
\end{thm}

\begin{proof}

First consider the case $\leftdes(r)\setminus I= \emptyset$. 
By definition, as $r$ is core, $I=\leftred(p)$. By Theorem \ref{thm:addremove}, no simple reflection can be added or removed from $I$ in a reduced expression for $r$, so $[I]$ is the unique reduced expression for $r$. This implies that $I=J$ and  that $\underline{r}=\id$. 

Let us now consider the case $ \leftdes(r)\setminus I \neq \emptyset$. If $s\in \leftdes(r)\setminus I$
then some reduced expression for $r$ has the form $[I,Is, \ldots, J]$. Let $n$ be the $(Is,J)$-coset such that $r = [I,Is] . n$.
The left redundancy of $n$ is at most the size of $J$, so it is a proper subset of $Is$. Again by Theorem \ref{thm:addremove}, for any $t \in Is \setminus \leftred(n)$, $n$ has a reduced expression of the form $[Is, Is \setminus t, \ldots J]$. Let $q$ be the $(Is \setminus t,J)$-coset such that $n = [Is,Is\setminus t] . q$.
Then
\begin{equation} r = [I,Is, Is \setminus t] . q, \end{equation}
as desired. Both $q$ and $p \expr [I + s -t]$ are core cosets, by Proposition \ref{prop:splitcore}. That  $t = w_{Is} s w_{Is}$, and that $[I + s - t]$ is the unique reduced expression for $p$, are implied by Lemma \ref{lem:atomic}.
\end{proof}




\begin{cor} \label{cor:coresplitsintoatomic} Any core $(I,J)$-coset $p$ has a reduced expression of the form 
\begin{equation}\label{eq:compositionatomics} I_{\bullet} = [I = I_0, I_1, \ldots, I_{2k}] = [I + s_1 - s_2 + s_3 - s_4 \ldots + s_{2k-1} - s_{2k}] \end{equation} for some $s_i \in S$. In particular, we have
\begin{equation} p = p_1 . p_2 . \cdots . p_k \end{equation}
where $p_i$ is a core coset with reduced expression $[I_{2i-2}, I_{2i-1}, I_{2i}]$. Moreover, $s_{2i} = w_{I_{2i-1}} s_{2i-1} w_{I_{2i-1}}$ for all $1 \le k$.
\end{cor}

\begin{proof} Iterate Theorem \ref{thm:corestartswithatomic}. \end{proof}

\begin{defn}
 A coset is \emph{atomic} if it is a core coset and has a reduced expression of the form $[I+s-t]$ for some $s, t \in S$. \end{defn}

Equivalently, by Lemma \ref{lem:atomic}, an $(I,J)$ coset $p$ is atomic if $J = Is \setminus t$ for $t = w_{Is} s w_{Is}$, and $p$ contains $w_{Is}$.

Recall that Theorem \ref{corehasatomic} stated that  any core coset is the reduced composition of atomic cosets.

\begin{proof}[Proof of Theorem \ref{corehasatomic}] This is a restatement of Corollary \ref{cor:coresplitsintoatomic}. \end{proof}

\subsection{Algorithm}

To conclude, here is an algorithm which produces a reduced atomic expression for any core $(I,J)$-coset $p$. That this algorithm works is effectively the proof of Theorem \ref{thm:corestartswithatomic}.

Choose an element $s \in \leftdes(\ma{p}) \setminus I$, and add it. Now subtract $t = w_{Is} s w_{Is}$. Let $I' = Is \setminus t$. Let $q$ be such that $p = [I + s - t] . q$. Explicitly, we let $q$ be the $(I', J)$-coset containing $w_{I'} w_{Is} \ma{p}$, justified since the definition of a reduced expression forces
\begin{equation} \ma{p} = (w_{Is} w_{I'}^{-1}) . \ma{q}. \end{equation}
Now repeat with the core $(I',J)$-coset $q$ replacing $p$.

This process terminates when $\leftdes(\ma{q}) = I'$, in which case $q$ is the identity coset.
The process is guaranteed to terminate by induction on the length of $\ma{p}$.

\subsection{Composition of atomic cosets}

\begin{lem} \label{reducedcompofatomiciscore}
    Let $p_1,\ldots, p_k$ be atomic cosets such that the composition $p_1*\ldots *p_k$ is well-defined and reduced. Then 
    $p_1*\ldots *p_k$ is a core coset.
\end{lem}
\begin{proof}
The case $k=2$ is a special case of \Cref{lem:mimimi}. The general case easily follows by induction.
\end{proof}

\begin{prop}\label{atomatom}
    Let $p$ be an atomic $(I,J)$-coset, and $q$ be an atomic $(J,K)$-coset.
    \begin{enumerate}
        
        \item\label{pqcore} If $p*q$ is not a reduced composition, then $p*q$ is a core coset if and only if $I=J=K$ and $p=q$, in which case we have $p*p=p$.
        \item\label{ppp} We have $p*(p^{-1})*p = p$.
    \end{enumerate}
\end{prop}

\begin{proof}
Let $p\expr[I+s-s']$ and $q\expr [J+t-t']$, so that $Js'=Is$ and $Kt'=Jt$, and suppose $p*q\expr[I+s-s'+t-t']$ is not reduced. We know that $[I+s-s']$ is reduced. Let $r$ be the coset expressed by $[I+s-s'+t]$.  If $[I+s-s'+t]$ were reduced, then $[I+s-s'+t-t']$ would also be reduced, since subtracting a simple reflection always preserves reduced expressions. Thus reducedness must fail at the step $+t$, see \cite[Definition 2.29]{EKo}. This can happen in one of two ways: either $\rightred(p) \subsetneq \rightred(r)$, or $\mi{p} > \mi{r}$. The former is impossible, since $\rightred(p) = J$ while the size of $\rightred(r)$ is at most the size of $I$, which is the size of $J$. Thus $\mi{p} \ne \mi{r}$.

\revise{As $\mi{p}$ is a minimal coset representative,} $\rightdes(\mi{p})$ is disjoint from $J$. Since $\mi{p}$ is not the identity element, $\rightdes(\mi{p})$ is nonempty. Since $\mi{p} \in W_{Js'}$ we must have $\rightdes(\mi{p}) = \{s'\}$. If $s' \ne t$ then $\rightdes(\mi{p}) \cap Jt$ would be empty, implying that $\mi{p} = \mi{r}$ is minimal in its $(I,Jt)$ coset, a contradiction. Thus $s' = t$.

Then 
\[ [I+s-s'+t-t'] = [I + s - t + t - t'] \expr [I+s-t']\] by the $*$-quadratic relation. This expresses a core coset if and only if $t'=s'$, which is if and only if $p=q$.

The last statement is straightforward. If $p \expr [I,M,J]$ then $\ma{p} = \ma{p^{-1}} = w_M$ and the result follows from $w_M * w_M = w_M$.
\end{proof}

\subsection{Atomic subcategories} \label{subsec:atomicsubcats}

In \S\ref{s:singmon} we recalled from \cite{EKo} the definition of the category $\SC = \SC(W,S)$, also called the singular Coxeter monoid. The objects are the finitary subsets of $S$, and the morphisms from $J$ to $I$ are $(I,J)$-cosets $p$, with composition operation $*$.

\begin{defn} Let $\AC = \AC(W,S)$ be the (non-full) subcategory of $\SC(W,S)$ generated by the atomic cosets. \end{defn}

Every core coset $p$ appears in the subcategory $\AC$ by Theorem \ref{corehasatomic}. Moreover, every reduced composition of atomic cosets is core by Lemma \ref{reducedcompofatomiciscore}. However, there are non-reduced compositions of atomic cosets which are not core, so that $\AC$ contains more than just the core cosets. Indeed, it is unclear which cosets have expressions as products of atomic cosets, and this makes $\AC$ fairly mysterious.

\begin{ex}
\revise{Let $W= S_3$ and $I=\{s_1\}\subset S=\{s_1,s_2\}$. There are two $(I,I)$-cosets, namely the identity coset $W_I$ (which is core) and the complement $p=W_I s_2 W_I$. The coset $p$ has  
$\leftred(p) = \emptyset$, so it is not a core coset, while its non-reduced expression 
\[[I,S,\{s_2\},S,I]=[I,S,\{s_2\}]\circ [\{s_2\},S,I]\] 
shows that $p$ belongs to $\AC$.} 
\end{ex}

\begin{ex} 
\revise{Let $W=S_4$ and $I=\{s_1,s_3\}\subset S = \{s_1,s_2,s_3\}$. There are three $(I,I)$-cosets $p,q,r$ with $\mi{p}=e, \mi{q}=s_2,\mi{r}=s_2s_3s_1s_2$. The cosets $p$ and $r$ are core and have atomic reduced expressions
\[p \expr [I],\quad r \expr [I,S,I].\]
The coset $q$ has a reduced expression
\[q\expr [[I\supset \emptyset \subset \{s_2\} \supset \emptyset \subset I]]\]
and is not core.
Since $r$ is the unique atomic $(I,J)$-coset in $W$ (for any $J\subset S$), and since $r* r = r$, the coset $q$ is not a composition of atomic cosets.}
\end{ex}

To isolate core cosets, we can shift our attention to a setting which does away with non-reduced compositions. In the traditional setting, the Coxeter monoid has generators $s \in S$ with quadratic relation $s * s = s$, while the nilCoxeter algebra has generators $\pa_s$ with quadratic relation $\pa_s \pa_s = 0$. The nilCoxeter algebra has the advantage that compositions of the generators associated to non-reduced compositions are zero. Returning to the world of double cosets, in \cite{paper1} we introduced a variant on $\SC$ called the nilCoxeter algebroid and denoted $\Dem = \Dem(W,S)$. The objects are finitary subsets of $S$, and the morphism space from $J$ to $I$ is the free $\Z$-module spanned by morphisms $\pa_p$ for $(I,J)$-cosets $p$. We have
\begin{equation} \label{eq:Demrelation} \pa_p \circ \pa_q = \begin{cases} \pa_{p * q} & \text{ if $p * q = p.q$ is reduced,} \\ 0 & \text{else.} \end{cases}\end{equation}

\begin{rem} In \cite{paper1} we provide several equivalent constructions of $\Dem$. It can be viewed as the associated graded of the linearization of $\SC$ with respect to a length filtration. It has a presentation by generators and relations, with the same generators and braid relations as $\SC$ but with a vanishing quadratic relation. Finally, it also appears as a non-full subcategory of vector spaces, encoding the action of push and pull operators on the cohomology rings of partial flag varieties. \end{rem}

\begin{rem}
In \cite{paper1} we defined $\Dem$ over a field, rather than over $\Z$. The equation \eqref{eq:Demrelation} gives a well-defined presentation over $\Z$. After base change to any field, this recovers the construction in \cite{paper1}, where it is proven that $\{\pa_p\}$ forms a basis. Therefore, $\{\pa_p\}$ forms a basis over $\Z$ as well.
\end{rem}


\begin{defn} Let $\AD = \AD(W,S)$ be the (non-full) subcategory of $\Dem$ generated by $\pa_p$ for atomic cosets $p$. \end{defn}

Unlike $\AC$, we have a precise handle on the size of $\AD$.

\begin{thm} The subcategory $\AD$ has a basis $\{\pa_p\}_{p \; \mathrm{is} \; \mathrm{core}}$. \end{thm}

\begin{proof} Let $\AD' \subset \Dem$ be the span of $\pa_p$ as $p$ ranges over all core cosets. Note that such $\pa_p$ then form a basis of $\AD'$ since $\pa_p$, for all double cosets $p$, form a basis for $\Dem$ (see \cite[Corollary 3.20]{paper1}). We have $\AD' \subset \AD$ by Theorem \ref{corehasatomic}, since $\pa_p$, for $p$ a core coset, is a (reduced) composition of atomic generators. If further 
$q$ is atomic, then either $p*q$ is core or $p*q$ is not reduced, by Lemma \ref{reducedcompofatomiciscore}. Thus $\AD'$ is closed under composition with the atomic generators, implying that $\AD \subset \AD'$. We conclude that $\AD = \AD'$. \end{proof}

\begin{rem} While $\Dem$ is the associated graded of the linearization of $\SC$, it is not true that $\AD$ is the associated graded of the linearization of $\AC$. As noted earlier, $\AC$ may contain non-core cosets, making it bigger than $\AD$. \end{rem}

\begin{rem} \label{rem:ACblocks} In the categories $\SC$ and $\Dem$, there are morphisms from $I$ to $J$ for any finitary subsets $I, J \subset S$. In $\AC$ and $\AD$, there can only be a morphism $I \to J$ when $I$ and $J$ are conjugate. Thus the categories $\AC$ and $\AD$ split into blocks, one for each conjugacy class. \end{rem}

\section{Atomic expressions in type A}\label{sec:typeA}

In this section, we explain how to handle atomic expressions in type $A$. Through an operation called "squashing," we establish a bijection between core cosets and permutations in smaller permutation groups. This is well-illustrated in the picture in \Cref{ex:squashing}, which the reader is encouraged to keep in mind. Moreover, in \S\ref{ssec:atomicbraid}, we describe how ordinary braid relations are transformed into compositions of singular braid relations between atomic expressions.

\subsection{Special notation in type \texorpdfstring{$A$}{A}}

Let $(W,S)=(\calS_n,\{s_1,\ldots,s_{n-1}\})$ be the Coxeter group of type $A_{n-1}$. 
We view $W$ as the permutation group $\calS(X)$ on the ordered set $X = \{1, \ldots, n\}$.
The length of an element $w\in \calS_n$ is then interpreted as the number of inversions, 
\begin{equation}\label{eq:lengthA}
    \ell(w) = \# \{(i,j)\ |\ 1\leq i<j\leq n, w(i)>w(j)\}.
\end{equation}

We use the common shorthand of referring to subsets of $S$ by their indices rather than their simple reflections. Thus $J = \{2,3,6\}$ or for short $J = 236$ represents the subset $\{s_2, s_3, s_6\}$ of $S$.  An intentional side effect of this shorthand is that $J$ can also be viewed as a subset of $X$, via the mapping $s_j \mapsto j$. For example, if $J \subset S$ and $w \in \calS_n$, we can consider both $w(J)$, a subset of $X$, and $w J w^{-1}$, a set of reflections in $\calS_n$. If $J$ appears within parentheses as in $w(J)$, then it is being viewed as a subset of $X$.



\subsection{Redundancy in type A}

The redundancy of a double coset in type A has an alternative description. This is best explained in terms of the string diagram, but let us put it in words first. 

\begin{prop}\label{redA}
Let $p$ be an $(I,J)$-coset, and let $j \in J$. Then
\begin{enumerate}
    \item\label{L=p.R} $\leftred(p)=\mi{p}(\rightred(p))$;
    \item\label{RRA} 
    $j\in\rightred(p)$ if and only if $ \mi{p}(j+1)=\mi{p}(j)+1$ and $\mi{p}(j)\in I$;
    \item\label{RRAcore}Assume further that $p$ is a core coset. Then  $j\in\rightred(p)$ if and only if $\mi{p}(j+1)=\mi{p}(j)+1$.
    
\end{enumerate}
\end{prop}

\begin{ex} Returning to Example \ref{ex:runningredunpic}, we have $3 \in \rightred(p)$ and $2 \in \leftred(p)$. We also have $\mi{p}(3) = 2$ and $\mi{p}(4) = 3$. Now considering the coset $p^{-1}$, note that $\mi{p}^{-1}(7) = 10$ and $\mi{p}^{-1}(8) = 11$ and $7 \in I$, but $7 \notin \leftred(p) = \rightred(p^{-1})$. This is not a violation of the last statement, since $p^{-1}$ is not a core coset. \end{ex}

\begin{proof}
    Let $j\in J$. Then by definition, we have $s_j\in\rightred(p)$ if and only if $\mi{p}s_j=s_i\mi{p}$ for some $i\in I$. Notice that in this case we have $i \in \leftred(p)$.
We have
\begin{equation}\label{j+1}
    \mi{p}(j+1)=\mi{p}s_j(j)=s_i\mi{p}(j)=s_i(\mi{p}(j)).
\end{equation}
If $i\not\in\{\mi{p}(j),\mi{p}(j)-1\}$ then \eqref{j+1} becomes $\mi{p}(j+1)=\mi{p}(j)$, a contradiction since $\mi{p}$ is a permutation.
If $i = \mi{p}(j)-1$ then we have $\mi{p}(j+1)<\mi{p}(j)$, and thus $\mi{p}s_j<\mi{p}$ by \eqref{eq:lengthA}. The latter contradicts that $\mi{p}$ is minimal in $\mi{p}W_J$. This proves $i=\mi{p}(j)$, which implies $\mi{p}(j+1)=s_i(\mi{p}(j)) = \mi{p}(j) + 1$. This justifies both statement~\eqref{L=p.R} and the `only if' part of statement~\eqref{RRA}.

If $\mi{p}(j) = i$ and $\mi{p}(j+1) = i+1$ then $\mi{p} s_j = s_i \mi{p}$. From this, the `if' part of statement~\eqref{RRA} follows.

If $p$ is core, then $I=\leftred(p)$ and $J=\leftred(p)$, so $\mi{p}(j)\in I$ directly follows from  \eqref{L=p.R}. Thus statement~\eqref{RRA} can be rewritten as statement~\eqref{RRAcore}.
\end{proof}

\begin{prop}\label{[j,k]s}
    Let $p$ be an $(I,J)$-coset. Then the connected components in $\leftred(p)$ and $\rightred(p)$ are the maximal subsets 
    \[ \{s_{i},s_{i+1}\ldots, s_{i+k}\}\subset I, \qquad \{s_j,s_{j+1},\ldots,s_{j+k}\}\subset J \] where $\mi{p}$ maps $\{j,j+1,\ldots, j+k,j+k+1\}$ to $\{i,i+1,\ldots, i+k,i+k+1\}$ in order-preserving fashion.
\end{prop}
\begin{proof}
    This follows easily from the characterization in Proposition~\ref{redA}\eqref{RRA}. For example, if $s_j,s_{j+1}\in\rightred(p)$ then $\mi{p}(j+2)=\mi{p}(j+1)+1=\mi{p}(j)+2$. 
\end{proof}

\subsection{Squashing}\label{s:w(p)}

 As is well-known, parabolic subgroups of $\calS_n$ have the form $\calS_{n_1} \times \cdots \times \calS_{n_k}$ with $$\sum_{c=1}^k n_c = n.$$ Thus there is a bijection between subsets of $S$ and contiguous set partitions of $X = \{1, \ldots, n\}$.

More pedantically, given $J \subset S$ with parabolic subgroup as above, let $\sim_J$ be the equivalence relation on $X$ generated by the relation $j \sim_J (j+1)$ whenever $s_j \in J$. The number of equivalence classes (denoted $k$ above) is $n - \# J$. The equivalence classes will be denoted  $C_c$ or $C_c^{(J)}$ for $1 \le c \le k$. Thus $C_1 = \{1, \ldots, n_1\}$. Let $X/\sim_J$ denote the set of equivalence classes under $\sim_J$. Each $C_c$ inherits an order from $X$, and the set $X/\sim_J$ itself inherits an order from $X$.

\begin{ex}\label{ex:CJ} Let $n=7$ and $J = \{s_2,s_3,s_6\}$. Then $X/\sim_J = \{C_1, C_2, C_3, C_4\}$ where
\[ C_1 = \{1\} \;\;  < \;\;  C_2 = \{2 < 3 < 4\} \;\;  < \;\; C_3 = \{5\} \;\;  <  \;\; C_4 = \{6 < 7\}. \] \end{ex}

Fixing a number $m$ and only considering subsets $J \subset S$ of size $m$, we identify $X/\sim_J$ as an ordered set with $Y_{n-m}: = \{1 < 2 < \ldots < n- m\}$.

\begin{defn}
    Fix $I$ and $J$ of size $m$. Suppose $y$ is a permutation of $X$ which sends each component $C_c^{(J)}$ via an order-preserving bijection onto a component $C_{c'}^{(I)}$. We say that $y$ is a \emph{block permutation} relative to $I$ and $J$. Clearly $y$ induces a bijection $X/\sim_J \to X/\sim_I$, which we can view as a permutation of $Y_{n-m}$. This permutation we denote $\squish(y) \in \calS(Y_{n-m}) \cong \calS_{n-m}$, and call it the \emph{squashed permutation} of $y$.
\end{defn}

\begin{lem} \label{lem:squishhomom} If $y$ is a block permutation relative to $I$ and $J$, and $z$ is a block permutation relative to $J$ and $K$, then $yz$ is a block permutation relative to $I$ and $K$. Moreover, if $m$ is the size of $I$ (and $J$ and $K$), then $\squish(y) \squish(z) = \squish(yz)$ within $\calS(Y_{n-m})$. \end{lem}

\begin{proof} This is obvious. \end{proof}

\begin{defn} \label{defn:squashcoset}
Let $p$ be a core $(I,J)$ coset. By Proposition \ref{[j,k]s}, $\mi{p}$ is a block permutation relative to $I$ and $J$. We call $\squish(p) := \squish(\mi{p})$ the \emph{squashed permutation} of $p$. 
\end{defn}

We refer to \Cref{ex:squash} for an illustration of Definition \ref{defn:squashcoset}.

\begin{thm}\label{thm.psi} Fix $J \subset S$ of size $m$. Then $\squish$ induces a bijection 
\begin{equation} \squish^J \co \{(I,p) \mid p \text{ is a core $(I,J)$-coset} \} \to \calS(Y_{n-m}). \end{equation}
\end{thm}
\begin{proof}
By Proposition \ref{[j,k]s}, the minimal element $\mi{p}$ of a core coset is determined by its induced permutation of $X/\sim_J$, i.e. by $\squish(\mi{p})$. Moreover, $I$ and $p$ are determined by $\mi{p}$, since $I = \mi{p} J \mi{p}^{-1}$. Thus $\squish^J$ is injective.

Surjectivity of $\squish^J$ follows because any permutation in $\calS(Y_{n-m})$ lifts to a block permutation $y$ of $\calS(X)$ relative to $J$. Setting $I = y(J)$, $y$ is the minimal element of a core $(I,J)$-coset.
\end{proof}

\begin{prop}\label{monoidmorphism?}
    If $p=p'.p''$ and $p,p',p''$ are core cosets in $\calS_n$, then we have $\squish(p)=\squish(p').\squish(p'')$. 
\end{prop}
\begin{proof}
    First we note that the assumption implies that the subsets 
    \[ \rightred(p),\leftred(p),\rightred(p'), \leftred(p'),\rightred(p''),\leftred(p'') \] all have the same size, say $m$. The equality $\squish(p)=\squish(p').\squish(p'')$ makes sense in $\calS(Y_{n-m})$. Now, if $p,p',p''$ are core cosets, then $p=p'.p''$ implies
    \[\mi{p}=\mi{p'}.\mi{p''}\]
    by \Cref{lem:mimimi}, and thus the claim follows from Lemma~\ref{lem:squishhomom}.
\end{proof}

Let us continue to fix $m$, and view $\calS(Y_{n-m})$ as a Coxeter group, with simple reflections $t_j$ which transpose $j$ and $j+1$.

\begin{prop} \label{prop:atomictosimple} If $p$ is an atomic coset in $\calS_n$, then $\squish(p)$ is a simple reflection. If $J$ is fixed, then $\squish^J$ induces a bijection between atomic $(I,J)$-cosets (as $I$ varies) and simple reflections in $\calS(Y_{n-\# J})$. \end{prop}

\begin{proof} Let $p$ be an atomic $(I,J)$ coset. Then there are simple reflections $s$ and $t$ such that $Is = Jt = M$ and $p \expr [I,M,J]$. Moreover, $s$ and $t$ are in the same component of $M$.

The difference between $\sim_J$ and $\sim_M$ is that two adjacent components are merged into one. More precisely, if $t$ is the $r$th simple reflection in $S\setminus J$, then 
 $C_r, C_{r+1} \in X/\sim_J$ are merged into $C_r \cup C_{r+1} \in X/\sim_M$. We know that $\ma{p} = w_M$. It is straightforward to deduce that $\mi{p}$ is the block permutation which swaps $C_r$ and $C_{r+1}$, so that $\squish(p) = t_r$. See Example \ref{ex:typeAatomic} for some illustrations.

For each simple reflection $t \in S \setminus J$, letting $M = Jt$ and $s = w_M t w_M$, there is exactly one atomic coset with reduced expression $[M \setminus s,M,J]$. 
The injectivity and surjectivity of $\squish^J$ (as a map between atomic cosets and simple reflections) is obvious.\end{proof}

\begin{thm}\label{thm.atomicrex} Let $p$ be a core $(I,J)$-coset. Then $\squish$ induces a bijection
\[\{ \text{atomic}  \text{ reduced expressions for $p$} \} \to \{\text{ordinary reduced expressions for }\squish(p) \}. \] \end{thm}

\begin{proof} 

By Proposition \ref{monoidmorphism?}, reduced compositions for core cosets are sent to reduced compositions, and by Proposition \ref{prop:atomictosimple}, atomic cosets are sent to simple reflections. Thus atomic reduced expressions for $p$ are sent to ordinary reduced expressions for $\squish(p)$. Injectivity of this map on reduced expressions follows from injectivity of squashing (\Cref{thm.psi}). That there is an atomic expression lifting any ordinary expression follows from the surjectivity in Proposition \ref{prop:atomictosimple}, and that it is a reduced atomic expression follows from Lemma \ref{lem:mimimi}.
\end{proof}

\subsection{Atomic braid relations in type \texorpdfstring{$A$}{A}} \label{ssec:atomicbraid}

Let $J\subset S$ and $\ell=\#(S\setminus J)$, and label $S \setminus J = \{d_1, \ldots, d_{\ell}\}$ in increasing order. For $s = s_{d_i} \in S\setminus J$ we write $\at^J_i$ for the atomic coset expressed by $[Js\setminus t,Js,J]$, where $t = w_{Js}sw_{Js}$. Then the atomic cosets ending in $J$ are precisely $\{\at^J_1, \ldots, \at^J_{\ell}\}$.


When $J$ is understood we omit the superscript and write $\at_i$. We denote by $\at_{i_d}*\cdots * \at^{J_0}_{i_0}$
the composition $\at^{J_d}_{i_d}*\cdots *\at^{J_0}_{i_0}$ where $J_d,\cdots, J_0$ is the unique sequence of superscripts such that the composition is well-defined.
We also suppress the superscript $J_0$ when any choice will do, as we do in the following lemma.

\begin{lem}\label{m3a}
The composition $\at_i*\at_{i\pm 1}*\at_i$ is reduced and thus core. Moreover, we have
\begin{equation}\label{eq:m3a}
\begin{alignedat}{2}
    \at_i.\at_{i+1}.\at_i &\quad \ \expr &&[I+s_a-s_b+s_c-s_d+s_b-s_e]\\
    &\underset{\text{switchback}}{\expr} &&[I+s_a +s_c -s_d -s_e ]\\
    &\overset{\text{downdown}}{\underset{\text{upup}}{\expr}} &&[I+s_c +s_a -s_e -s_d]\\
    &\underset{\text{switchback}}{\expr} &&[I+s_c-s_f+s_a-s_e+s_f-s_d] \expr \at_{i+1}.\at_i.\at_{i+1}.
    \end{alignedat}
\end{equation}
Here the middle $`\expr$' is the composition of two commuting braid relations, and $I,a,b,c,d,e,f$ are determined by $i$ and the hidden superscript $J_0$.
\end{lem}

\begin{proof}
To ease the notation we assume $i=1$ and let $x=|C_1^{(I)}|=|C_3^{(J_0)}|$, $y=|C_2^{(I)}|=|C_2^{(J_0)}|$, $z=|C_3^{(I)}|=|C_1^{(J_0)}|$. 
The general case is obtained from this by shifting the indices.

We have 
\[\at_1*\at_2*\at_1 \expr [I+ s_{x}-s_y+s_{x+y}-s_{y+z} +s_{y}-s_{z}],\]
where the switchback relation $[-s_y+s_{x+y}-s_{y+z} +s_{y}] \expr [s_{x+y}-s_{y+z}]$ applies (see \cite[Subsection 6.1]{EKo}).
Similarly, we have 
\[\at_2*\at_1*\at_2= [I+ s_{x+y}-s_{x+z}+s_{x}-s_{z} +s_{x+z}-s_{y+z}],\]
where the switchback relation $[-s_{x+z}+s_{x}-s_{z} +s_{x+z}] \expr [s_{z}-s_{x}]$ applies.
Composing with the upup and downdown relations, we complete the relation \eqref{eq:m3a}. 

Since the middle two expressions are reduced (any multistep expression of the form $[[I\subset K\supset J]]$ is reduced) and the relations are braid relations, all expressions in \eqref{eq:m3a} are reduced, confirming the first claim.
\end{proof}

\begin{lem}\label{m2a}
For $|i-j|>1$, the composition $\at_i*\at_j$ is reduced and thus core. We have
    \begin{equation}\label{eq:m2a}
    \begin{alignedat}{2}
        \at_i.\at_j&\quad \ \expr && [I+s_a-s_b+s_c-s_d]  \\
        &\underset{\text{switchback}}{\expr} && [I+s_a+s_c-s_b-s_d] \\
        &\overset{\text{downdown}}{\underset{\text{upup}}{\expr}}&&[I+s_c+s_a-s_d-s_b]  \\
        &\underset{\text{switchback}}{\expr} &&[I+s_c-s_d+s_a-s_b]\expr \at_j.\at_i,
    \end{alignedat}
    \end{equation}
    where the middle $`\expr$' is the composition of two commuting braid relations and $I,a,b,c,d$ are determined by $i$ and the hidden superscript $J_0$.
\end{lem}
\begin{proof}
Let $c,d,J$ be such that $\at_j^{J_0}\expr [J+s_c-s_d]$ and let $a,b,J'$ be such that $\at_i^{J_0}\expr[J'+s_a-s_b]$. It is easy to verify that $\at_i^{J} = [I + s_a - s_b]$ and $\at_j^{J'} = [I + s_c - s_d]$ for the same parabolic subset $I$. If $|i-j|>1$, then  $s_b$ and $s_c$ (resp., $s_a$ and $s_d$) belong to different connected components of $J\sqcup s_c$ (resp., $J'\sqcup s_a$) and thus the claimed switchback relations hold, see \cite[\S 5.6]{EKo}. 
The last equality is straightforward.

Since the middle two expressions are reduced and the involved relations are braid relations, all expressions in \eqref{eq:m3a} are reduced, confirming the first claim. 
\end{proof}

We call \emph{atomic braid relations} in type $A$ the two relations 
\[ \at_i.\at_{i+1}.\at_i\expr \at_{i+1}.\at_{i}.\at_{i}\qquad \at_i.\at_j\expr \at_j.\at_i \text{ if } |i-j|>1\]
between atomic reduced expression
that we have just proved in \Cref{m2a} and \Cref{m3a}. We obtain an atomic version of the Matsumoto theorem.
 
\begin{thm}[Matsumoto's theorem for atomic reduced expressions in type  $A$]\label{thm.atomicMatsumoto}
    Let $p$ be a core $(I,J)$-coset. Then any two atomic reduced expressions of $p$ are related by atomic braid relations.
\end{thm}
\begin{proof}
\Cref{thm.atomicrex} tells us that atomic reduced expressions for $p$ correspond to ordinary reduced expressions for $\squish(p)$. By Matsumoto's theorem on ordinary expressions, all reduced expression for $\squish(p)$ are related by braid moves, so also atomic reduced expressions for $p$ are related by atomic braid relations. 
\end{proof}

\begin{rem} 
Theorem~\ref{thm.atomicMatsumoto} does not say that every relation between atomic reduced expressions of a core $(I,J)$-coset $p$ is a composition of the relations of the form \eqref{eq:m2a} and \eqref{eq:m3a}. Such a composition factors only through the expressions involving the subsets of $S$ whose size is between $\# J$ and $\# J +2$.  For example, when $S \setminus J$ has size at least $3$, and $\squish(p)$ is the longest element, one can find a reduced expression $I_{\bullet}$ for $p$ which factors through $[[I \subset S]]$, and a sequence of singular braid moves between atomic reduced expressions which factors through $I_{\bullet}$. \end{rem}

\subsection{Non-reduced atomic expressions}

\begin{lem}\label{aa=a}
Let  $\at_i^J\expr [I,Js,J]$ be an atomic coset.
\begin{enumerate}
    \item The composition $\at^I_i*\at^J_i$ is not reduced and $\at^I_i*\at^J_i\cong [J,Js,J]$.
    \item  The composition $\at^I_i*\at^J_i$ is core if and only if $I=J$, in which case we have $\at_i*\at_i=\at_i$.
    \item We have $\at_i*\at_i*\at_i^J=\at_i^J$.
\end{enumerate}
\end{lem}
\begin{proof}
We have $\at_i^J\expr [I+s'-s]$ where $s'$ is  the $i$-th element in $S\setminus I$.
By definition, we can write $\at^I_i\expr [K,It,I]$, where $t$ is the $i$-th element in $S\setminus I$. Thus, we have $s'=t$.

Since $\at^I_i$ is atomic, we can also write $\at^I_i\expr [K+t'-t]$ with $t'=w_{It}tw_{It}$. So
\[\at^I_i*\at^J_i \expr[K+t'-t]*[I+s'-s]=[K+t'-t+s'-s]\]
contains the reducible expression $[-t+t]\expr []$.
This shows that the composition is not reduced and we have $\at_i^I*\at_i^J\expr [K+t'-s]$.

Moreover, we also have have $s=w_{Is'}s'w_{Is'}=w_{It}tw_{It}=t'$, so also $s=t'$ and $K=J$. This shows the first statement.

The second statement follows from \Cref{atomatom}\eqref{pqcore} and
the third follows from \Cref{atomatom}\eqref{ppp} since $\at_i^I\expr [J,Js,I]\expr (\at_i^J)^{-1}$.
\end{proof}

\begin{rem}
    If $\at_i*\at_i$ is not core, then we have $\at_i*\at_i \expr [J,Js,J]$ but $s\neq w_{Js} s w_{Js}$.
\end{rem}

\subsection{Atomic categories in type \texorpdfstring{$A$}{A}}

Recall from \S\ref{subsec:atomicsubcats} the definitions of the categories $\AC$ and $\AD$. We now discuss these categories in type $A$.

The braid relations of \S\ref{ssec:atomicbraid} and the non-braid relations from Lemma \ref{aa=a} can be interpreted as relations in $\AC$. They induce relations in $\AD$. More precisely, the braid relations can be copied to $\AD$ verbatim (after a change in notation), while Lemma \ref{aa=a} implies that $\pa_{\at_i} \circ \pa_{\at_i} = 0$.

We will discuss $\AD$ first, as it is easier to understand, and will resume the discussion of $\AC$ afterwards.
Our first goal is to present $\AD$ by generators and relations, akin to typical descriptions of the KLR algebra \cite{KLDiagrammatic}. Indeed, $\AD$ is simply an algebroid version of the usual nilCoxeter algebra.

Let $y$ be a block permutation relative to $I$ and $J$. We write $J \leftrightarrow (n_1, \ldots, n_k)$ when $W_J \cong \calS_{n_1} \times \cdots \times \calS_{n_k}$. After squashing $y$ to a permutation on $k$ strands, we can label the strands with the numbers $n_1$ through $n_k$ as in the following example.
\begin{equation} \label{labelednilcoxex} \squish \co \quad \ig{1}{examplecoreatomicrex} \quad \mapsto \quad {
\labellist
\small\hair 2pt
 \pinlabel {$1$} [ ] at 0 -5
 \pinlabel {$1$} [ ] at 8 -5
 \pinlabel {$3$} [ ] at 16 -5
 \pinlabel {$2$} [ ] at 24 -5
 \pinlabel {$1$} [ ] at 32 -5
 \pinlabel {$2$} [ ] at 40 -5
 \pinlabel {$1$} [ ] at 48 -5
 \pinlabel {$1$} [ ] at 0 90
 \pinlabel {$3$} [ ] at 8 90
 \pinlabel {$1$} [ ] at 16 90
 \pinlabel {$2$} [ ] at 24 90
 \pinlabel {$1$} [ ] at 32 90
 \pinlabel {$1$} [ ] at 40 90
 \pinlabel {$2$} [ ] at 48 90
\endlabellist
\centering
\ig{1}{examplesquashed}
} \end{equation}
In this way we can view the image of $\squish^J$ as permutations of labeled strands, with the bottom labels determined by $J$.

\begin{defn} Let $n \ge 0$. The \emph{labeled nilCoxeter algebroid} $\nilcox(n)$ is the $\Z$-linear category presented as follows. The objects are compositions of $n$, i.e. sequences $(n_1, \ldots, n_k)$ with $\sum n_c = n$, identified with parabolic subsets $J$ as above. We draw the identity map $1_J$ of a composition as a sequence of labeled strands.
The generating morphisms are crossings $\pa_i^J = \pa_i 1_J$, which are morphisms from $J \leftrightarrow (n_1, \ldots, n_i, n_{i+1}, \ldots, n_k)$ to $(n_1, \ldots, n_{i+1}, n_i, \ldots, n_k)$. The relations are 
\begin{equation}\label{Nrel1}
    \partial_i\partial_{i+1}\partial_i^J=\partial_{i+1}\partial_{i}\partial_{i+1}^J \quad \text{for $1\leq i\leq k-2$}
\end{equation}
    \begin{equation}\label{Nrel2}
        \partial_i\partial_j^J = \partial_j\partial_i^J \quad \text{ if $|i-j|>1$}
    \end{equation}
    \begin{equation}\label{Nrelnil}
        \partial_i\partial_i^J = 0 \text{ for $1\leq i\leq k-1$}.
    \end{equation}
We have omitted superscripts when they are uniquely determined.
\end{defn}

The RHS of \eqref{labelednilcoxex} now represents the morphism in $\nilcox(11)$:
\[ \pa_5 \pa_6 \pa_4 \pa_5 \pa_3 \pa_4 \pa_2 \co (1,1,3,2,1,2,1) \to (1,3,1,2,1,1,2).\] 

\begin{prop}\label{cor.AC0}

We have isomorphisms of $\Bbbk$-linear algebroids
    \begin{align*}
         \psi \co \nilcox(n) & \xrightarrow{\sim} \AD(\calS_n)\\
          \partial_i^J&\mapsto \partial_{\at_i^J}.
    \end{align*}
\end{prop}

\begin{proof}
The relations ~\eqref{Nrel1} and \eqref{Nrel2} hold in $\AD$ by \Cref{m3a} and \Cref{m2a} respectively. The relation~\eqref{Nrelnil} also follows by \Cref{aa=a}, since non-reduced compositions are zero in $\AD$. Thus the functor $\psi$ is well-defined.

Fix $J \subset S$, corresponding to a composition with $k$ parts. Let $p$ be a core $(I,J)$-coset. We know that $\pa_p$ is equal to a composition of $\pa_{\at_i}$ over an atomic reduced expression for $p$. By Theorem \ref{thm.psi}, $\psi$ sends $\squish^J(p)$, interpreted as a product of the generators $\pa_i$ in $\nilcox(n)$, to $\pa_p$ in $\AD$. Thus $\psi$ is surjective. By Theorem \ref{thm.psi}, $\psi$ will be injective if the left ideal
\[ \bigoplus_I \Hom_{\nilcox(n)}(J,I) = \nilcox(n) \cdot 1_J \]
has a $\Z$-basis of the form $\{\squish^J(p)\}$, in bijection with permutations of $k$.

Let $\mathcal{N}(k)$ denote the ordinary nilCoxeter algebra on $k$ strands. There is a bijection between $\mathcal{N}(k)$ and $\nilcox(n) \cdot 1_J$. Indeed, these spaces have exactly the same presentation by generators and relations. This proves that the size of $\nilcox(n) \cdot 1_J$ is exactly as desired.
\end{proof}


Finding a presentation for $\AC$ is more difficult. The relation 
\begin{equation}\label{eq.atomcub}
    \at_i*\at_i*\at_i^I =\at_i^I
\end{equation}
holds in general, replaced by
\begin{equation}\label{eq.atomquad}
    \at_i^I*\at_i^I=\at_i^I,
\end{equation}
when $\at_i^I$ is an $(I,I)$-coset. We are unsure whether these relations, combined with the braid relations, are sufficient to describe $\AC$ in type $A$.

In the special case where $n = k \ell$, letting $J$ correspond to the composition $(k,k, \ldots,  k)$ of $n$, one sees that there are no morphisms between $J$ and $I$ for any $I \ne J$. The monoid $\AC(J) = \End_{\AC}(J)$ agrees with the $*$-monoid for $\calS_{\ell}$, by a very similar argument to Proposition \ref{cor.AC0} above but with one additional input coming from \cref{aa=a}. Namely, if $p$ is a core $(J,J)$-coset squishing to a permutation with $s_i$ in its right descent set, then $p$ has an atomic reduced expression ending in $\at_i^J$, and thus \eqref{eq.atomquad} implies that $p * \at_i^J = p$. We leave the proofs to the reader.

\section{Atomic expressions in type \texorpdfstring{$B$}{B}}\label{s.typeB}

Having belabored the point in type $A$, let us skip to the chase in type $B$. We indicate how to squash a core coset in type $B$ into an element of a smaller type $B$ Weyl group. Under this squashing procedure, atomic cosets are sent to simple reflections. We indicate how each of the braid relations in type $B$ lifts to a sequence of singular braid relations between atomic expressions. We do not include many proofs, as they are reasonably adaptable from the proofs in type $A$.

Let $(W_n,S)$ be of type $B_n$ with $S =\{s_0,s_1,\ldots, s_{n-1}\}$, where $m_{s_0 s_1} = 4$ and $m_{s_i s_{i+1}} = 3$ otherwise. Write $W = W_n$. The only conjugate of $s_0$ which is a simple reflection is $s_0$ itself. Therefore, if $p$ is an $(I,J)$-coset then $s_0 \in \rightred(p)$ if and only if $s_0 \in \leftred(p)$, if and only if $s_0 \in I \cap J$ and $s_0$ commutes with $\mi{p}$. 

View $W$ as the group of permutations $\sigma$ of the set $X = \{-n, \ldots, -1, 0, 1, \ldots, n\}$ such that $\sigma(-i) = - \sigma(i)$. For any such $\sigma$ we have $\sigma(0) = 0$, but we include the index $0$ to simplify some discussion later.

In this setting, we have $s_i=(i \;\; i+1) (-i\;\; -i-1)$ for $i>0$ and $s_0=(-1\;\; 1)$.
Note that $y \in W$ commutes with $s_0$ if and only if $y$ preserves the set $\{ \pm 1\}$.

For $J \subset S$, let $\sim_J$ be the equivalence relation on $X$ generated by the relations $j \sim_J (j+1)$ and $-j \sim_J (-j-1)$ whenever $s_j \in J$. The equivalence classes are clearly contiguous subsets of $X$, so if $i < 0 < j$ and $i \sim_J j$ then $i \sim_J 0 \sim_J j$.

\begin{lem} Fix $J \subset S$. Let $C_0$ be the equivalence class of $0$ under $\sim_J$. If $i \in C_0$ then $-i \in C_0$. Any other equivalence class consists of either only positive or only negative numbers. If $C$ is such an equivalence class, then $-C = \{-i \mid i \in C\}$ is also an equivalence class. \end{lem}

We label the equivalence classes under $\sim_J$ in order as 
\[ \{C_{-k}, \ldots, C_{-1}, C_0, C_1, \ldots, C_k\} \] where $C_{-c} = - C_c$ and $C_0$ contains $0$. When $\# J = m$ is fixed, we write $Y_{n-m}$ for the ordered set $\{-k, \ldots, -1, 0, 1, \ldots, k\}$ which indexes the equivalence classes under $\sim_J$. Note that $k = n-m$.

Fix $I$ and $J$ of size $m$. We call $y \in W$ a \emph{block permutation} relative to $I$ and $J$ if each component $C_c^{(J)}$ is sent by an order-preserving bijection onto a component $C_{c'}^{(I)}$. If so, $y$ induces a bijection $X/\sim_J \to X/\sim_I$, which we can view as a permutation $\squish(y)$ of $Y_{n-m}$. We call this the \emph{squashed permutation} of $y$.

\begin{ex} Here is a block permutation in $W_5$ relative to $I = \{s_1, s_3\}$ and $J = \{s_1, s_4\}$, and its squashed permutation. The components $C$ are marked with rectangles, and the index $0$ has a thicker strand.
\begin{equation} \squish \co \quad \ig{1}{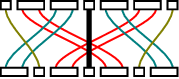} \quad \mapsto \quad \ig{1}{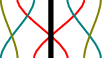}\; .
\end{equation}
\end{ex}

\begin{ex} Here is a block permutation in $W_5$ relative to $I = \{s_0, s_1, s_3\}$ and $J = \{s_0, s_1, s_4\}$, and its squashed permutation.

\begin{equation} \squish \co \quad \ig{1}{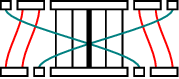} \quad \mapsto \quad \ig{1}{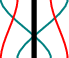}\; .
\end{equation}

\end{ex}

\begin{lem} The permutation $\sigma = \squish(y)$ of any block permutation $y$ satisfies the property that $\sigma(-i) = - \sigma(i)$ for all $i \in Y_{n-\# J}$. Therefore, $\squish(y)$ lives in the Weyl group of type $B_{n-\# J}$. \end{lem}

\begin{lem} If $y$ is a block permutation relative to $I$ and $J$, and $z$ is a block permutation relative to $J$ and $K$, then $yz$ is a block permutation relative to $I$ and $K$. Moreover, if $m$ is the size of $I$ (and $J$ and $K$), then $\squish(y) \squish(z) = \squish(yz)$ within the Weyl group of type $B_{n-m}$. \end{lem}

\begin{lem} Let $p$ be a core $(I,J)$-coset. Then $\mi{p}$ is a block permutation relative to $I$ and $J$. \end{lem}

We write $\squish(p) := \squish(\mi{p})$. Now we summarize the properties of $\squish$, just as in type $A$.

\begin{thm}\label{thm:atomicrexB}  Fix $J \subset S$ of size $m$, and let $k = n-m$.  Let $W_k$ be the Weyl group of type $B_k$, acting on $Y_k$. If $p = p' . p''$ is a reduced composition of core cosets in $W_n$, then $\squish(p) = \squish(p').\squish(p'')$ is a reduced composition in $W_k$. The map $\squish$ induces a bijection 
\begin{equation} \squish^J \co \{(I,p) \mid p \text{ is a core $(I,J)$-coset} \} \to W_k, \end{equation}
which restricts to bijection between atomic $(I,J)$-cosets (as $I$ varies) and simple reflections in $W_k$.  Consequently, for any core $(I,J)$-coset $p$, $\squish^J$ sends atomic reduced expressions for $p$ bijectively to ordinary reduced expressions for $\squish(p)$.
\end{thm}

More explicitly, label $S \setminus J = \{s_{d_0}, \ldots, s_{d_{k-1}}\}$ in increasing order. For $s = s_{d_i} \in S\setminus J$ we write $\at^J_i$ for the atomic coset expressed by $[Js\setminus t,Js,J]$, where $t = w_{Js}sw_{Js}$. Then $\squish^J$ sends $\at^J_i$ to the simple reflection $s_i$ of $W_k$. Notice that \Cref{aa=a} still holds in this setting with the same proof.

We give now the sequence of double coset braid relations in $W_n$ which correspond to each braid relation in $W_k$.  The atomic braid relations $\at_i.\at_{i+1}.\at_i\expr \at_{i+1}.\at_i.\at_{i+1}$ for $i>0$ and $\at_i.\at_j\expr \at_j.\at_i$ for $|i-j|> 1$ are obtained as in type $A$ (cf. \Cref{m3a} and \Cref{m2a}). The remaining atomic braid relation can be obtained as follows.
\begin{equation}
\begin{alignedat}{2}
    \at_1.\at_0.\at_1.\at_0 &\quad\ \expr && [I+s_a(-s_b+s_c-s_d+s_e-s_f+s_g)-s_h] \\
    & \underset{\text{switchback}}{\expr} &&[I+s_a+s_c -s_f-s_h ]\\
    &\overset{\text{downdown}}{\underset{\text{upup}}{\expr}}&&
    [I+s_c+s_a-s_f-s_h]\\
    & \underset{\text{switchback}}{\expr} &&[I+s_c(-s_i+s_a -s_j+s_k -s_f+ s_l)-s_h]\expr \at_0.\at_1.\at_0.\at_1.
\end{alignedat}
\end{equation}

We conclude by giving a presentation of the atomic category $\AD(W_n,S)$ by generator and relations. Proofs are analogous to the type $A$ setting. We have chosen to identify parabolic subgroups $J$ by pinning down the size of its components $C_i$ with $i \ge 0$.

\begin{defn} Let $n \ge 0$. The \emph{labeled nilCoxeter algebroid of type $B$}, denoted $\nilcox_B(n)$, is the $\Z$-linear category presented as follows. The objects are those sequences $(n_0,n_1 \ldots, n_k)$ of positive integers such that $\sum n_c = n+1$. These can be identified with parabolic subsets $J\subset S$.
The generating morphisms are $\pa_i^J = \pa_i 1_J$, for $0\leq i\leq k-1$. If $i>0$, then $\pa_i^J$ is a morphisms from $J \leftrightarrow (n_0, n_1, \ldots, n_i, n_{i+1}, \ldots, n_k)$ to $(n_0, n_1, \ldots, n_{i+1}, n_i, \ldots, n_k)$ for $i>0$, while $\pa_0^J$ is a morphism from $J$ to itself. Note that $n_0$ is never changed by a morphism. The relations are
\begin{equation}    \pa_0\pa_1\pa_0\pa_1^J=\pa_1\pa_0\pa_1\pa_0^J
\end{equation}
\begin{equation}
\partial_i\partial_{i+1}\partial_i^J=\partial_{i+1}\partial_{i}\partial_{i+1}^J \quad \text{for $1\leq i\leq k-2$}
\end{equation}
    \begin{equation}
        \partial_i\partial_j^J = \partial_j\partial_i^J \quad \text{ if $|i-j|>1$}
    \end{equation}
    \begin{equation}
        \partial_i\partial_i^J = 0 \text{ for $0\leq i\leq k-1$}.
    \end{equation}
We have omitted superscripts when they are uniquely determined.
\end{defn}

\begin{prop}\label{prop:ADpresentationB}
We have isomorphisms of $\Z$-linear algebroids
    \begin{align*}
         \psi \co \nilcox_B(n) & \xrightarrow{\sim} \AD(W_n,S)\\
          \partial_i^J&\mapsto \partial_{\at_i^J}.
    \end{align*}
\end{prop}

\sloppy
\printbibliography
\end{document}